\documentclass{article}

\usepackage[letterpaper,top=2cm,bottom=2cm,left=3cm,right=3cm,marginparwidth=1.75cm]{geometry}

\usepackage{amsmath,amsfonts,amssymb,amsthm,booktabs,color,epsfig,graphicx,hyperref,url}

\usepackage[
    backend=biber,
    maxbibnames=99
]{biblatex}
\usepackage{enumerate}
\usepackage{enumitem}
\usepackage{fullpage}

\theoremstyle{plain}
\newtheorem{Thm}{Theorem}[section]
\newtheorem{Pro}[Thm]{Proposition} 
\newtheorem{Lem}[Thm]{Lemma}
\newtheorem{Ass}{Assumption}

\newtheorem{Cor}[Thm]{Corollary}
\newtheorem{Eg}[Thm]{Example}
\newtheorem{Rem}[Thm]{Remark}
\newtheorem{Def}[Thm]{Definition}

\newcommand{\cl}{\mathcal}

\newcommand{\mbf}{\mathbf}
\newcommand{\bb}{\mathbb}

\newcommand{\E}{\mathbb{E}}
\newcommand{\Cov}{\mathrm{Cov}}
\newcommand{\Var}{\mathrm{Var}}
\newcommand{\EqD}{\overset{d}{=}}

\def\pp#1{ \left(#1\right) }
\def\pb#1{ \left[#1\right] }
\def\pc#1{ \left\{#1\right\} }
\def\Prt#1{ \mathbb{P} \pp{#1} }

\date{}

\title{Conditional Independence under Infinite Measures and Poisson Point Processes}
\author{
Shuyang Bai\footnote{Corresponding Author.}\\
bsy9142@uga.edu\\
University of Georgia
\and
Vishal Routh\\
v.routh@uga.edu\\
University of Georgia
}

\begin{document}
\maketitle

\begin{abstract}
We study conditional independence under infinite measures on punctured product spaces, a notion recently introduced for graphical modeling in multivariate extremes and L\'evy processes. In contrast to classical probabilistic conditional independence, this concept is formulated through normalized restrictions of an infinite measure, a formulation that reflects the non-product structure of the punctured space. We show that this non-standard notion admits a natural probabilistic characterization: it is equivalent to classical conditional independence between coordinate projections of a Poisson point process defined on the punctured   space with the given infinite measure as its mean measure. In addition, we provide a functional characterization of the conditional independence concept at the level of the enumerated points of the Poisson point process. We further extend the characterization to punctured products of Borel  spaces with distinguished points and a  product localization.
\end{abstract}


\section{Introduction and overview of results}\label{Sec:intro}

Recently, a notion of conditional independence under an infinite measure defined on the punctured product space 
$\mathbb{R}^V \setminus \{0_V\}$, $V=\{1,\ldots,d\}$,  $d \in \mathbb{Z}_+=\{1,2,\ldots\}$, where  $0_V$ denotes the origin of 
$\mathbb{R}^V$,  has been developed by 
\cite{engelke2025graphical}, extending the pioneering work of 
\cite{engelke2020graphical}. 
Owing to the infinite total mass of the measure 
and the non-product structure of the underlying space, the formulation of this 
conditional independence is necessarily non-standard and differs from classical 
probabilistic conditional independence. We recall the main definition as follows.  

Write
$$E_V^o:=\mathbb{R}^V \setminus \{0_V\}.$$ 
Throughout the paper, a superscript circle $\circ$ indicates that the corresponding object is associated with a punctured space. 
Suppose $\Lambda^o_V$ is a Borel measure on $E_V^o$
 satisfying
\[
\Lambda^o_V(G) < \infty 
\quad \text{for all } G \in \cl{B}(\mathbb{R}^V) 
\text{ such that } 0_V \notin \overline{G},
\]
where $\overline{G}$ denotes the closure of $G$, and we write $\cl{B}(T)$ for the Borel $\sigma$-field of a 
topological space $T$ here and below.
Define
\begin{equation}\label{eq:R test class}
\mathcal{R}\left(\Lambda^o_V\right)
=\pc{R=\prod_{v\in V} R_v : 
R_v \in \cl{B}(\mathbb{R}),\ \Lambda^o_V(R)>0,\ 
0_V \notin \overline{R}}.
\end{equation}
Each element of $\mathcal{R}\left(\Lambda_V^o\right)$ is a product-form subset of 
$E_V^o$ with finite $\Lambda^o_V$-mass. 
The idea is to assess classical probabilistic conditional independence 
on such subsets under normalized restrictions of $\Lambda^o_V$. 

For $A \subseteq V$ and $y \in \mathbb{R}^V$, we write 
$y_A = (y_v)_{v \in A}$ for the corresponding subvector.

\begin{Def}[\cite{engelke2025graphical}, Definition 3.1]\label{Def:Lambda cond indep}
For disjoint subsets $A, B, C \subseteq V$ forming a partition of $V$, 
we say that $\Lambda_V^o$ admits the conditional independence relation
\[
A \perp B \mid C 
\]
if for every $R \in \mathcal{R}\left(\Lambda_V^o\right)$ the classical probabilistic 
conditional independence
\[
Y_A \perp Y_B \mid Y_C
\]
holds, where $Y = (Y_v)_{v \in V}$ is a random vector taking values in 
$R$ with probability law $P_R$ defined by
\[
P_R(G) := \frac{\Lambda_V^o(G)}{\Lambda_V^o(R)}, 
\qquad G \in \cl{B}\left(\bb{R}^V\right), \ G\subset R.
\]

If $C = \emptyset$,  we understand the relation as 
unconditional independence $A\perp B$. If $A = \emptyset$ or $B = \emptyset$, 
the conditional or unconditional independence relation is understood to hold 
trivially.

If the disjoint sets $A$, $B$, and $C$ do not form a partition of $V$, 
the definition is then applied to $\Lambda_D^o$, where $D = A \cup B \cup C$, 
and $\Lambda_D^o$ denotes the marginal measure induced by $\Lambda_V^o$ on 
$E_D^o:=\mathbb{R}^D \setminus \{0_D\}$.
\end{Def}
Equivalent characterizations of this definition can be found in 
\cite{engelke2025graphical}, including one involving a reduced class 
of test sets from $\mathcal{R}\left(\Lambda_V^o\right)$ 
\cite[Theorem~4.1]{engelke2025graphical}, and another based on a 
probability kernel factorization 
\cite[Theorem~4.4]{engelke2025graphical}.  See also Section \ref{Sec:graph} below.

Furthermore, \cite{engelke2025graphical} introduced the following explosiveness condition, denoted there as Assumption~(E1):
\begin{equation}\label{eq:E1}
  \Lambda_V^o\!\left(y_d\neq 0,\, y_{V\setminus D}=0_{V\setminus D}\right)\in\{0,\infty\},
  \qquad \text{for every } D\subset V \text{ and } d\in D.
\end{equation}
Note that assumption \eqref{eq:E1} implies that $\Lambda_V^o\pp{E_V^o}$ is either $0$ or $\infty$, and only the latter case is nontrivial.
See also Lemma~\ref{Lem:E1 equiv} below for an equivalent formulation.
 With the condition \eqref{eq:E1} additionally imposed, the conditional independence relation in Definition \ref{Def:Lambda cond indep}   satisfies the so-called \emph{semigraphoid axioms}: suppose $A,B,C,D$ are disjoint subsets of $V$.
\begin{enumerate}[label=(L\arabic*),leftmargin=*,itemsep=0.3em]

\item \emph{Symmetry.} 
If $A \perp B \mid C $, then 
$B \perp A \mid C $;

\item \emph{Decomposition.} 
If $A \perp (B \cup D) \mid C $, then 
$A \perp B \mid C $ and 
$A \perp D \mid C $;

\item \emph{Weak union.} 
If $A \perp (B \cup D) \mid C $, then 
$A \perp B \mid (C \cup D) $;

\item \emph{Contraction.} 
If $A \perp B \mid C $ and 
$A \perp D \mid (B \cup C) $, then 
$A \perp (B \cup D) \mid C $.

\end{enumerate}

The conditional independence notion introduced in Definition 
\ref{Def:Lambda cond indep} has already begun to play a fundamental 
role in the statistical graphical modeling of multivariate extremes 
and multivariate L\'evy processes; see, for example, the recent survey 
article \cite{engelke2024graphical} and the references therein for 
the former, and \cite{engelke2024evy, bruck2026graph} for the latter.

The main objective of this work is to  show that the  non-standard notion of conditional 
independence  can, in fact, 
be characterized in terms of a classical probabilistic conditional 
independence relation between the coordinate projections of a 
Poisson point process on the punctured space $E_V^o$. 
This characterization yields a   transparent interpretation of 
the original concept.

 Suppose $\xi^o_V$ is a Poisson point process on $E_V^o$ with mean measure $\Lambda_V^o$ that is the same  as the one described before Definition \ref{Def:Lambda cond indep}. See Section \ref{Sec:prelim} for more details about the background of Poisson point processes.
 For     $A\subset V$, we write $\xi_{A}$ for  the marginal Poisson point process induced by $\xi_V^o$ on the non-punctured space $\bb{R}^A$, that is, $\xi_A(G)=\xi_V^o(y_A\in G)$, $G\in \cl{B}(\bb{R}^A)$.     We have the following characterization.

\begin{Thm}\label{Thm:intro}
Suppose the assumption \eqref{eq:E1} holds,
 and we have disjoint subsets $A, B, C \subseteq V$. 
Then the conditional independence relation
\[
A \perp B \mid C 
\]
in Definition \ref{Def:Lambda cond indep} holds if and only if
\[
\xi_A\perp \xi_B \mid \xi_C.
\]
 Here, when $C = \emptyset$,  the relation is understood as unconditional independence, and when $A=\emptyset$ or $B=\emptyset$, the conditional or unconditional independence relation  is understood to hold trivially.  
\end{Thm}
The theorem follows as a special case of Theorem \ref{Thm:characterization cond indep} (see also Remark \ref{Rem:charac}) below. We note that a similar characterization, formulated in terms of the jump part 
of a multivariate L\'evy process, has been established in 
\cite{engelke2024evy}, where $\Lambda_V^o$ in Definition 
\ref{Def:Lambda cond indep} corresponds to the L\'evy measure. 
Nevertheless, the characterization in Theorem \ref{Thm:intro} should be 
distinguished from characterizations based on pure-jump L\'evy processes. 
In particular, the latter framework involves a space-time Poisson point 
process whose spatial projection is a Poisson point process with mean 
measure proportional to $\Lambda_V^o$. Conditioning in that setting is of a 
different nature, since repeated spatial locations can be distinguished 
through their temporal marks. 
Furthermore, a L\'evy measure  needs to satisfy a second-moment-type integrability condition near the origin, rendering the class of admissible measures 
more restrictive than that considered in Theorem \ref{Thm:intro}.

Recall that for random vectors $X_A$, $X_B$ and $X_C$, the usual probabilistic conditional independence relation $X_A\perp X_B\mid X_C$ is equivalent to the functional representation (see, e.g., \cite[Proposition 8.20 and Lemma 4.22]{kallenberg:2021:foundations})
\begin{equation}\label{eq:class cond indep fun rep}
(X_A,X_B,X_C)\EqD  \pp{h_A(X_C,\theta_A), h_B(X_C,\theta_B), X_C } 
\end{equation}
for some measurable functions $h_A$ and $h_B$, where $\theta_A$ and $\theta_B$ are i.i.d.\ Uniform($0,1$) random variables independent of $X_C$, and  $\EqD $  stands for equality in distribution. The functional characterization can alternatively be stated in an asymmetric form involving only one of the uniform randomizers, say $\theta_A$ or $\theta_B$. For instance, one may write
$
(X_A,X_B,X_C)\EqD \bigl(h_A(X_C,\theta_A),\,X_B,\,X_C\bigr),
$
with the requirement that $\theta_A$ be independent of $(X_B,X_C)$. In what follows, however, we work with the symmetric formulation \eqref{eq:class cond indep fun rep}, as the asymmetric version is recovered as a special case.

It turns out that we can formulate a functional characterization of the conditional independence in Definition \ref{Def:Lambda cond indep} at the level of the enumerated points of the Poisson point process as well. 
\begin{Pro}\label{Pro:intro}
Under the setup of Theorem \ref{Thm:intro}, suppose in addition that    $A,B,C\subset V$ are each nonempty and form a partition of $V$, and that $\Lambda_V^o\pp{E_V^o}=\infty$.
 The conditional independence in Theorem \ref{Thm:intro} holds if and only if for some measurable functions $h_A:\bb{R}^C\times [0,1]\to \bb{R}^A $ and $h_B:\bb{R}^C\times [0,1]\to \bb{R}^B $ satisfying $h_A(0_C, \cdot )\equiv 0_A$ and $h_B(0_C, \cdot )\equiv 0_B$, we have
 \[
 \xi_V^o  \EqD \sum_{i=1}^{\kappa} \delta_{\pp{ h_A\pp{\eta_{iC},\theta_{iA}}+\eta_{iA}, \, h_B\pp{\eta_{iC},\theta_{iB}}+\eta_{iB}, \,  \eta_{iC}}},
 \]
 where $\sum_{i=1}^{\kappa} \delta_{\eta_i}=\sum_{i=1}^{\kappa} \delta_{\pp{\eta_{iA},\eta_{iB},\eta_{iC}}}$ is a Poisson point process on $E_V^o$ with total point count $\kappa$ and mean measure $\Lambda^\perp_{AB;C}$ defined by 

\[
\begin{aligned}
\Lambda^\perp_{AB;C}(G)
:={}&\Lambda_V^o\bigl(G\cap\{y_B=0_B,\,y_C=0_C\}\bigr)\\
&+\Lambda_V^o\bigl(G\cap\{y_A=0_A,\,y_C=0_C\}\bigr)\\
&+\Lambda_C^o\bigl(\{y_C\in E_C^o:(0_A,0_B,y_C)\in G\}\bigr),
\end{aligned}
\]
where $ G\in \cl{B}\pp{E_V^o}$, with $\Lambda_C^o$ denoting the marginal measure induced by $\Lambda_V^o$ on 
$E_C^o:=\mathbb{R}^C \setminus \{0_C\}$,
 and $\pp{\theta_{iA}}$ and   $\pp{\theta_{iB}}$ are two mutually independent i.i.d.\ Uniform$(0,1)$ sequences that are independent of  $\sum_{i=1}^{\kappa} \delta_{\eta_i}$.
\end{Pro}
We provide some   explanation of the measure $\Lambda^\perp_{AB;C}$ and the functional representation here.
When viewed as an exponent measure in multivariate extreme value theory, the definition of $\Lambda^\perp_{AB;C}$ enforces extremal independence among the three components $A$, $B$, and $C$, in the sense that at most one of them can be nonzero; see \cite{engelke2025graphical} for further discussion. 
To interpret the functional representation, note that when $\eta_{iC}=0_C$, the functions $h_A$ and $h_B$ both vanish, and at most one of $\eta_{iA}$ and $\eta_{iB}$ can be nonzero. This corresponds to extremal independence between the $A$ and $B$ components when the $C$ component vanishes.   When $\eta_{iC}\neq 0_C$ instead, both $\eta_{iA}$ and $\eta_{iB}$ must vanish, yielding a representation that resembles the probabilistic counterpart \eqref{eq:class cond indep fun rep}. 
Proposition \ref{Pro:intro} follows as a special case of Proposition \ref{Pro:cond indep punc} below.

We also extend Definition~\ref{Def:Lambda cond indep} and its Poisson point process characterization to punctured products of Borel spaces. This extends the coordinate spaces while retaining the single-puncture geometry and the product localization generated by bounded sets in individual coordinates.

The rest of the paper is organized as follows. In Section \ref{Sec:prelim}, we review preliminary concepts related to localized Borel spaces and Poisson point processes. Section \ref{Sec:cond indep} investigates conditional independence of projections of a Poisson point process on non-punctured and punctured product spaces. In Section \ref{Sec:graph}, we discuss the generalization of the conditional independence concept from Definition \ref{Def:Lambda cond indep}, along with its characterization in terms of Poisson point processes.

\section{Preliminaries}\label{Sec:prelim}

In this section, we recall some  background  related to   Poisson point processes. The main references are \cite{LastPenrose2017} and \cite[Chapter 15]{kallenberg:2021:foundations}.  

Let $(E,\cl{E})$ be a measurable space. Recall that a measure on 
$(E,\cl{E})$ is said to be $s$-finite if it is a 
countable sum of finite measures. Note that a measure may be $s$-finite 
without being $\sigma$-finite, a distinction that is important in our context.
Let $\cl{N}(E)$ denote the space of point measures on $(E,\cl{E})$, 
where a  point measure  is an $s$-finite 
$\overline{\bb{N}}_0$-valued measure, where
$$\overline{\bb{N}}_0=\{0,1,\ldots,\infty\}=\{0\}\cup\bb{Z}_+\cup \pc{\infty}.
$$  We equip the space $\cl{N}(E)$ with the smallest $\sigma$-field that makes  the evaluation map $\mu\mapsto \mu(A)$  measurable for each fixed $A\in \cl{E}$.

A \emph{Poisson point process} (or \emph{Poisson random measure}) $\xi$ with an $s$-finite mean measure $\Lambda$ on $(E,\cl{E})$ is a random element taking values in $\cl{N}(E)$,  
satisfying:
\begin{enumerate}[label=(\roman*)]
\item For every $A \in \mathcal{E}$, the random variable $\xi(A)$ 
has a Poisson distribution with mean $\Lambda(A)$,  where if $\Lambda(A)=0$ or $\infty$, the random variable $\xi(A)$ is understood as $0$ or $\infty$  respectively.

\item[(ii)] For every $m \in \mathbb{Z}_+$ and all   disjoint sets 
$A_1,\dots,A_m \in \mathcal{E}$, the random variables 
$\xi(A_1),\dots,\xi(A_m)$ are independent.
\end{enumerate}
For any $s$-finite measure $\Lambda$ on $(E,\cl{E})$, there exists a Poisson point process $\xi\in \cl{N}(E)$ with mean measure $\Lambda$ \cite[Theorem 3.6]{LastPenrose2017}.

For a measurable function $f:E\to\bb{R}$,  it is known (e.g., \cite[Proposition 12.1]{LastPenrose2017}) that the Poisson random integral $\int  f(y) \xi(dy)$ exists a.s.\ under the condition
\begin{equation}\label{eq:pois int cond}
     \int   \pp{|f(y)|\wedge 1} \Lambda(dy)<\infty.
\end{equation}
The following uniqueness result, which will be used in a proof later, is likely  known. However, since we have been unable to locate an exact reference, we include a brief proof for completeness.
\begin{Lem}[Uniqueness]\label{Lem:uniq}
Under the setup above, for a measurable function $f:E\to \bb{R}$ satisfying the integrability condition \eqref{eq:pois int cond}, we have $\int f(y) \xi(dy)=0$ a.s.\ if and only if $f(y)=0$ for $\Lambda$-a.e.\ $y\in E$.
\end{Lem}
\begin{proof}
The ``if'' part is trivial. We prove the ``only if'' part.

Write $f=f_+-f_-$ with $f_+=f \mbf{1}_{\{f>0\}}$ and $f_-= -f \mbf{1}_{\{f\le 0\}}$. Under the assumption \eqref{eq:pois int cond},
both Poisson  integrals $\int f_+\,d\xi$ and $\int f_-\,d\xi$ are finite a.s., and
\[
\int f\,d\xi=\int f_+\,d\xi-\int f_-\,d\xi.
\]
The assumption $\int f\,d\xi=0$ a.s.\ therefore implies
\begin{equation}\label{eq:pos neg equal}
\int f_+\,d\xi=\int f_-\,d\xi \quad \text{a.s.} 
\end{equation}
Due to the disjoint supports of the integrands $f_+$ and $f_-$,  the two integrals above are independent of each other. Hence,  both are necessarily the same   constant a.s., say $c\in [0,\infty)$.    Applying the Laplace functional formula of the Poisson point process (e.g., \cite[Theorem 3.9]{LastPenrose2017}),  we have for all $t\ge 0$,
\[
 tc=-\log\mathbb E\exp\Big\{-t\int f_+(y)\xi(dy)\Big\}
=\int (1-e^{-t f_+(y)})\,\Lambda(dy).
\]
The function $t\mapsto 1-e^{-t f_+(y)}$  is strictly concave  if $f_+(y)>0$. So if $\Lambda\pp{f_+>0}>0$, then the last  integral above is a strictly concave function in $t$, contradicting the first linear expression above. We thus infer $f_+=0$ $\Lambda$-a.e. The same argument also applies to $f_-$.

\end{proof}

In many applications, one works with a measurable space $(E,\cl{E})$ 
that is sufficiently regular to ensure the existence of a measurable 
enumeration of the point masses of a Poisson point process.

Recall that $(E,\cl{E})$ is called a \emph{Borel space} if there exists 
a bijection between $E$ and a Borel subset of $[0,1]$ such that both 
the bijection and its inverse are measurable. In particular, every 
singleton subset of $E$ belongs to $\cl{E}$. Moreover, any complete separable metric space (or more generally, a Polish space), equipped with its Borel $\sigma$-algebra, is a Borel space \cite[Theorem 1.8]{kallenberg:2021:foundations}.

In addition, we require an abstract notion of boundedness. If there 
exists an increasing sequence of sets $(L_h)_{h \in \bb{Z}_+}$ with 
$L_h \in \cl{E}$, $L_h \subset L_{h+1}$, and 
$\bigcup_{h \in \bb{Z}_+} L_h = E$, then $(E,\cl{E})$ is said to be 
\emph{localized} by $(L_h)_{h \in \bb{Z}_+}$. Under such a localization, 
a set $A \in \cl{E}$ is called \emph{bounded} if $A \subset L_h$ for 
some $h \in \bb{Z}_+$, and a measure $\mu$ on $(E,\cl{E})$ is called 
\emph{locally finite} if $\mu(A) < \infty$ for every bounded set $A$. 
In particular, every locally finite measure is $\sigma$-finite, and 
hence also $s$-finite.
Throughout, local finiteness is understood relative to the specified
localization. It may depend on this choice, but is unchanged when the
localizing sequence is replaced by one defining the same bounded sets.

For a localized Borel space $(E,\cl{E})$ with localizing sequence 
$(L_h)_{h \in \bb{Z}_+}$, define
\[
\cl{N}_L(E) 
:= \{\mu \in \cl{N}(E) : \mu(L_h) < \infty \text{ for all } h \in \bb{Z}_+ \},
\]
that is, the measurable subspace of $\cl{N}(E)$ consisting of all 
locally finite point measures. If a Poisson point process $\xi\in \cl{N}(E)$ on 
$(E,\cl{E})$ has a locally finite mean measure, then   $\xi\in \cl{N}_L(E)$ a.s.

By an \emph{enumerated representation} of $\mu \in \cl{N}_L(E)$, we 
mean the expression
\[
\mu = \sum_{i=1}^{\mu(E)} \delta_{\pi_i(\mu)},
\]
where 
$\pi_i : \cl{N}_L(E) \to E$, $i \in \bb{Z}_+$, are   measurable maps. For the existence   
of such representations, see   
\cite[Corollary 6.5]{LastPenrose2017}.
In particular, applying the enumerated representation to a Poisson point process 
$\xi$ with a locally finite mean measure $\Lambda$ yields
\begin{equation}\label{eq:enumerated rep}
\xi = \sum_{i=1}^{\kappa} \delta_{Y_i},
\end{equation}
where $\kappa$ is a Poisson random variable with mean $\Lambda(E)$, and $(Y_i)_{i \in \bb{Z}_+}$ is a sequence of $E$-valued random elements. In particular, the $\sigma$-field generated by $\xi$ is equal to the $\sigma$-field generated by $\pp{\kappa,\pp{Y_{i}}_{i\in \bb{Z}_+}}$. If the exact equality    in \eqref{eq:enumerated rep} is relaxed to equality 
in distribution, the existence of such an enumerated representation 
(on a possibly different probability space) requires only that 
$(E,\cl{E},\Lambda)$ is an $s$-finite measure space \cite[Theorem 3.6]{LastPenrose2017}.

A useful way to relate different Poisson point processes is through 
probability kernel transforms. Suppose $(F,\cl{F})$ is another 
measurable space, and let $\nu(\cdot \mid \cdot)$ be a probability 
kernel from $E$ to $F$, that is, $\nu(\cdot \mid y)$ is a probability 
measure on $(F,\cl{F})$ for each $y \in E$, and for every 
$B \in \cl{F}$ the mapping $ \nu(B \mid \cdot)$ is 
$\cl{E}$-measurable.
Let $\xi$ be a Poisson point process on $(E,\cl{E})$ with   enumerated representation as in \eqref{eq:enumerated rep}.
The \emph{$\nu$-transform} of $\xi$ is the point process
\[
\zeta := \sum_{i=1}^{\kappa} \delta_{Z_i},
\]
where, conditionally on $\kappa$ and $(Y_i)_{i \in \bb{Z}_+}$, the $F$-valued 
random elements $(Z_i)_{i \in \bb{Z}_+}$ are independent and each 
$Z_i$ has conditional distribution $\nu(\cdot \mid Y_i)$.
Then it is known  \cite[Theorem 5.6]{LastPenrose2017} that  $\zeta$ is a Poisson point process on $(F,\cl{F})$ with mean 
measure
\[
M(B) 
:= \int  \nu(B \mid y)\,\Lambda(dy),\quad B\in \cl{F},
\]
where $M$ is $s$-finite whenever $\Lambda$ is $s$-finite.

\section{Conditional independence of projections of Poisson point processes}\label{Sec:cond indep}

In this section, we examine conditional independence relations 
between coordinate projections of a Poisson point process. Since a single 
conditional independence relation involves three components, we 
restrict attention, without loss of generality, to a 
three-dimensional setting. Section~\ref{Sec:graph} discusses the 
higher-dimensional framework, where multiple conditional 
independence relations interact.

We begin by considering the non-punctured product space and then 
extend the discussion to the punctured case, building on the 
former analysis.
\subsection{Non-punctured product space}
 
Suppose $(E_j,\cl{E}_j)$, $j=1,2,3$, are measurable spaces. Let
\[
\pp{E_{123}:=E_1\times E_2\times E_3,\,  \cl{E}_{123}:=\mathcal E_1\otimes\mathcal E_2\otimes\mathcal E_3} 
\]
be the product measurable space. Let $\Lambda_{123}$ be a measure defined on $(E_{123},\cl{E}_{123})$, and let $\Lambda_j$ be the induced marginal measure on $E_j$ given by
  $$\Lambda_j(A)=\Lambda_{123}\pp{\{y_j\in A   ,\ (y_1,y_2,y_3)\in E_{123}\}},\quad A  \in \cl{E}_j ,\quad j=1,2,3.
  $$  

We need  the following set of regularity assumptions on these measurable spaces.
 \begin{Ass}\label{Ass:product}
 ~
\begin{enumerate}[label=(\roman*)]

\item  The measurable spaces $(E_j,\cl{E}_j)$, $j=1,2,3$, are  Borel. 

\item   The measurable spaces $(E_3,\cl{E}_3)$ and  $ (E_{123},\cl{E}_{123}) $ are localized, and the measures  $\Lambda_{123}$ and $\Lambda_3$ are    locally finite.

\item Each  $\Lambda_j$, $j=1,2$,  is $\sigma$-finite when restricted to $E_j\setminus\{o_j\}$ for a single point $o_j\in E_j$. 
\end{enumerate}
\end{Ass}
\begin{Rem}\label{Rem:ass 1}
We emphasize that in Assumption \ref{Ass:product} (iii), the relaxation about the single point    is essential, and will play a key role in our subsequent analysis of punctured spaces.
\end{Rem}

By disintegration  (e.g., \cite[Theorem 3.4(iii)]{kallenberg:2021:foundations}), there exists a probability   kernel $\Lambda_{12\mid 3}(\cdot\mid \cdot):\pp{\cl{E}_1\otimes \cl{E}_2} \times E_3\to [0,1] $, such that
\[
\Lambda_{123}(A) = \int_{}\int \mbf{1}_A(y_1,y_2,y_3) \Lambda_3(dy_3)  \Lambda_{12\mid 3}(dy_1 dy_2|y_3), \quad A\in \cl{E}_{123}.
\]
We also introduce the marginal probability kernels 
\[
\Lambda_{1\mid 3}(\cdot \mid \cdot )=\Lambda_{12\mid 3}(\cdot \times E_2 \mid \cdot ), \quad \Lambda_{2\mid 3}(\cdot\mid\cdot)=\Lambda_{12\mid 3}( E_1\times \cdot  \mid \cdot ),
\]
and define the probability kernel $\Lambda_{123\mid 3}(\cdot\mid \cdot)$: $\cl{E}_{123}\times E_3\to [0,1]$, 
\[
\Lambda_{123\mid 3}(A\mid y_3)=\int  \mbf{1}_A(y_1,y_2,y_3)   \Lambda_{12\mid 3}(dy_1 dy_2|y_3), \quad  A\in \cl{E}_{123}, \ y_3\in E_3.  
\]

Now let $\xi_{123}$ be a Poisson point process on $E_{123}$ with mean measure $\Lambda_{123}$.
In the same manner, we define its projected point processes on $E_j$ by
\begin{equation}\label{eq:Lambda_j}
\xi_j(A)
:=
\xi_{123}\big(\{(y_1,y_2,y_3)\in E_{123}:\ y_j\in A\}\big),
\qquad A\in\mathcal E_j,
\quad j=1,2,3,
\end{equation}
whose corresponding mean measures are $\Lambda_j$, $j=1,2,3$.

The following conditional moment formula will be useful.
\begin{Lem}\label{Lem:cond mean cov}
Suppose Assumption \ref{Ass:product} holds,  $A_j\in \cl{E}_j$ and $\Lambda_j(A_j)<\infty$, $j=1,2$. Then  the following  almost-sure equalities hold:
\[
\E[\xi_j(A_j)\mid\xi_3]= \int  \Lambda_{j\mid 3}(A_j \mid y_3 ) \xi_3(dy_3), \quad j=1,2,
\]
and
\begin{align*}
&\Cov\pp{\xi_1(A_1), \xi_2(A_2)\mid \xi_3}\\= &\int \pb{
\Lambda_{12\mid 3}(A_1\times A_2\mid y_3)
-
\Lambda_{1\mid 3}(A_1\mid y_3)\,
\Lambda_{2\mid 3}(A_2\mid y_3)}
\xi_3(dy_3).
\end{align*}
\end{Lem}
\begin{proof}
 Note that the Poisson point process $\xi_{123}$ is equal in distribution to  a $\Lambda_{123|3}$-transform of $\xi_3$ (see Section \ref{Sec:prelim}).
Applying the conditional Laplace functional formula  (e.g.,\ \cite[Lemma~15.2(iii)]{kallenberg:2021:foundations}) with function
$f(y_1,y_2,y_3):=t_1\mathbf 1_{A_1}(y_1)+t_2\mathbf 1_{A_2}(y_2)$, we have the conditional joint Laplace transform:
\begin{equation}\label{eq:Phi-def}
\Phi(t_1,t_2):=\E\big[e^{-t_1\xi_1(A_1)-t_2 \xi_2(A_2)}\mid \xi_3\big]
=
\exp\left\{\int  \log \psi(t_1,t_2;y_3) \xi_3(dy_3)\right\},
\end{equation}
 where $t_1,t_2\in [0,\infty)$ and
\begin{align*}
&\psi(t_1,t_2;y_3) :=   \int e^{-t_1\mathbf 1_{A_1}(y_1)-t_2\mathbf 1_{A_2}(y_2)}\,
\Lambda_{12\mid 3}(dy_1dy_2\mid y_3)\\=&\int \bigl[1+(e^{-t_1}-1)\mathbf{1}_{A_1}(y_1)\bigr]
\bigl[1+(e^{-t_2}-1)\mathbf{1}_{A_2}(y_2)\bigr] \Lambda_{12\mid 3}(dy_1dy_2\mid y_3) 
\\= & 1
+(e^{-t_1}-1)\Lambda_{1\mid3}(A_1\mid y_3)
+(e^{-t_2}-1)\Lambda_{2\mid3}(A_2\mid y_3)\\
&+(e^{-t_1}-1)(e^{-t_2}-1)\Lambda_{12\mid3}(A_1\times A_2\mid y_3).
\end{align*}
The first-order partial derivatives   $-\partial_{t_1}\log \Phi(0,0)$ and $-\partial_{t_2}\log \Phi(0,0)$  yield the formulas for the
conditional means. The mixed derivative   $\partial_{t_1}\partial_{t_2}\log\Phi(0,0)$  yields the formula for the
conditional covariance. The differentiations under the random integral $\int \cdot \ \xi_3(dy_3)$ in \eqref{eq:Phi-def} can be justified by the dominated convergence theorem with the mean value theorem through conditioning on $\xi_3$. In particular,  using the fact that $\psi(t_1,t_2;y_3)\ge e^{-t_1-t_2}$  as well as  some elementary calculations, one can show that for $(t_1,t_2)$ in a fixed neighborhood of $(0,0)$, the following  bounds  hold for all $y_3\in E_3$:
\[
\bigl|\partial_{t_1}\log\psi(t_1,t_2;y_3)\bigr|
\le C\,\Lambda_{1\mid3}(A_1\mid y_3),\qquad
\bigl|\partial_{t_2}\log\psi(t_1,t_2;y_3)\bigr|
\le C\,\Lambda_{2\mid3}(A_2\mid y_3),
\]
and
\begin{equation}\label{eq:mix partial bound}
\bigl|\partial_{t_1t_2}\log\psi(t_1,t_2;y_3)\bigr|
\le C\Big(\Lambda_{12\mid3}(A_1\times A_2\mid y_3)
+\Lambda_{1\mid3}(A_1\mid y_3)\Lambda_{2\mid3}(A_2\mid y_3)\Big),
\end{equation}
where  $C\in (0,\infty)$ is a constant independent of $y_3$ and uniform over the fixed neighborhood of $(t_1,t_2)$. The right-hand sides above
are all integrable with respect to $\xi_3$ a.s.; for example, for the last term in \eqref{eq:mix partial bound},  
\[
 \int
\Lambda_{1\mid3}(A_1\mid y_3)\Lambda_{2\mid3}(A_2\mid y_3) \Lambda_3(dy_3) \le  \int
\Lambda_{1\mid3}(A_1\mid y_3)  \Lambda_3(dy_3)=\Lambda_1(A_1)<\infty,
\]
and  thus verifies the integrability condition   with respect to $\xi_3$ (see \eqref{eq:pois int cond}).

\end{proof}

We are now ready to state the main result of this section: a characterization of conditional independence of marginal Poisson point processes. 

\begin{Pro}\label{Pro:cond indep prod}
Suppose Assumption \ref{Ass:product} holds,  and $\xi_{123}$ is a Poisson point process  with  mean measure $\Lambda_{123}$ on $(E_{123},\cl{E}_{123})$ and  projected Poisson point processes $\xi_{j}$, $j=1,2,3$, as in \eqref{eq:Lambda_j}.    The following statements are equivalent.
 \begin{enumerate}[label=(\alph*)]
     \item  The conditional independence relation  $\xi_1\perp \xi_2 \mid \xi_3$ holds.
     \item  For $\Lambda_3$-a.e.\ $y_3\in E_3$, the product measure factorization holds:
     \[
     \Lambda_{12\mid 3}(\cdot \mid y_3)= \Lambda_{1\mid 3}(\cdot\mid y_3)\otimes \Lambda_{2\mid 3}(\cdot\mid y_3).
     \]
     \item  Let $\sum_{i=1}^\kappa \delta_{\eta_i}$ be an enumerated representation of $\xi_3$. Then there exist    measurable functions $h_1:E_3\times [0,1]\to E_1$   and $h_2: E_3\times [0,1]\to E_2$, so that 
     \begin{equation}\label{eq:func rep cond indep}
     \xi_{123} \EqD 
    \sum_{i=1}^\kappa \delta_{\pp{h_1(\eta_i,\theta_{i1}),h_2(\eta_i,\theta_{i2}),\eta_i} },   
     \end{equation}
 where   $\pp{\theta_{i1}}_{i\in \bb{Z}_+}$ and $\pp{\theta_{i2}}_{i\in \bb{Z}_+}$  are two i.i.d.\ $\mathrm{Uniform}(0,1)$ sequences which are mutually independent and also independent of $\sum_{i=1}^\kappa \delta_{\eta_i}$. 
 \end{enumerate}
\end{Pro}
\begin{proof}
    We argue  the  cycle (a)$\Rightarrow$(b)$\Rightarrow$(c)$\Rightarrow$(a).

\noindent  $\bullet$ (a)$\Rightarrow$(b).

For each $j=1,2$, choose a countable generating $\pi$-system
$\mathcal C_j^*$ for $\mathcal E_j$ containing $E_j$.
By Assumption~\ref{Ass:product}(iii), there are increasing measurable
sets $B_{h,j}\uparrow E_j\setminus\{o_j\}$ with
$\Lambda_j(B_{h,j})<\infty$. Then
\[
\mathcal C_j
=
\{D\cap B_{h,j}:D\in\mathcal C_j^*,\ h\in\mathbb Z_+\}
\]
is a countable $\pi$-system of finite $\Lambda_j$-measure sets.
It generates $\mathcal E_j$: it contains every $B_{h,j}$ and hence
generates $\{o_j\}$, while
$\bigcup_h(D\cap B_{h,j})=D\setminus\{o_j\}$ for every
$D\in\mathcal C_j^*$.

Now take $A_j\in \cl{C}_j$, $j=1,2$.  By the conditional covariance formula in Lemma \ref{Lem:cond mean cov}, we see that
\[
 \int
\Lambda_{12\mid 3}(A_1\times A_2\mid y_3)
-
\Lambda_{1\mid 3}(A_1\mid y_3)\,
\Lambda_{2\mid 3}(A_2\mid y_3)
\,\xi_3(dy_3)=0 \text{ a.s.}
\]
Applying the uniqueness Lemma \ref{Lem:uniq}, we conclude that 
\[
\Lambda_{12\mid 3}(A_1\times A_2\mid y_3)
=\Lambda_{1\mid 3}(A_1\mid y_3)\,
\Lambda_{2\mid 3}(A_2\mid y_3)
\quad\text{for }\Lambda_3\text{-a.e. }y_3.
\]
By countability, the preceding identity holds simultaneously for all
$A_j\in\mathcal C_j$, $j=1,2$, outside a single $\Lambda_3$-null set.
It also holds when $A_1=E_1$ or $A_2=E_2$, by the definitions of the
marginal kernels. Thus the two probability measures agree on the
generating $\pi$-system of rectangles with factors in
$\mathcal C_j\cup\{E_j\}$, $j=1,2$, and Dynkin's theorem proves (b).

\medskip

\noindent$\bullet$ (b)$\Rightarrow$(c).

By \cite[Lemma 4.22]{kallenberg:2021:foundations},   there exist measurable functions $h_1:E_3\times [0,1]\to E_1$   and $h_2: E_3\times [0,1]\to E_2$, such that for all $y_3\in E_3$ and $i\in \bb{Z}_+$
\[
\Prt{h_1(y_3,\theta_{i1})\in\cdot} = \Lambda_{1\mid 3}(\cdot\mid y_3), \quad  \Prt{h_2(y_3,\theta_{i2})\in\cdot} = \Lambda_{2\mid 3}(\cdot\mid y_3).
\]
Now, under the product measure factorization assumption of (b),  we have
\[
\Prt{\pp{h_1(y_3,\theta_{i1}),h_2(y_3,\theta_{i2})}\in \cdot }=\Lambda_{12\mid 3}(\cdot\mid y_3),\quad   \text{for $\Lambda_3$-a.e. }y_3\in E_3.
\]
The conclusion now follows from the fact that a $\Lambda_{123\mid 3}$-transform of $\xi_3$ is equal in distribution to  $\xi_{123}$.
\medskip

\noindent$\bullet$  (c)$\Rightarrow$(a)

From the assumptions in (c), one may assume that
\[
\xi_1= \sum_{i=1}^\kappa \delta_{h_1(\eta_i,\theta_{i1})}, \quad  \xi_2= \sum_{i=1}^\kappa \delta_{h_2(\eta_i,\theta_{i2})},\quad  \xi_3=  \sum_{i=1}^\kappa \delta_{\eta_i}. 
\]
The  conditional independence claim in (a) follows  if  
 one identifies the following almost sure representations: 
\begin{equation}\label{eq:functional rep}
\xi_1=H_1(\xi_3,\Theta_1),\quad \xi_2=H_2(\xi_3,\Theta_2),
\end{equation}
where the two random elements $\Theta_1:=\pp{\theta_{i1}}_{i\in \bb{Z}_+}$, $\Theta_2:=\pp{\theta_{i2}}_{i\in \bb{Z}_+}$ are mutually independent and also independent of $\xi_3$, with each $H_j$  a measurable map $\cl{N}_L(E_3)\times[0,1]^{\bb{Z}_+}\to \cl{N}(E_j)$, $j=1,2$. Indeed, these maps can be constructed by suitable  compositions of measurable maps, by noticing the measurability of the enumeration map $\cl{N}_{L}(E_3)\to   \overline{\bb{N}}_0 \times  E_3^{\bb{Z}_+}$, $ \xi_3\mapsto \pp{\kappa, \pp{\eta_i}_{i\in\bb{Z}_+} }$, as well as the measurability of the assembling map  $\overline{\bb{N}}_0 \times E_j^{\bb{Z}_+}\to \cl{N}(E_j) $, $\pp{\kappa, \pp{Y_{ij}}_{i\in \bb{Z}_+}}\mapsto \sum_{i=1}^{\kappa} \delta_{ Y_{ij}}$, with $Y_{ij}=h_j(\eta_i,\theta_{ij})$, $j=1,2$.

\end{proof}

\begin{Rem}
We also note that, possibly after extending the underlying probability space, 
one can construct on the same probability space as $\xi_{123}$ a point process given 
by the right-hand side of \eqref{eq:func rep cond indep} that coincides with 
$\xi_{123}$ a.s.; see \cite[Lemma 4.21 and Corollary 8.18]{kallenberg:2021:foundations}.
\end{Rem}


\subsection{Punctured product spaces}

Let $(E_j,\mathcal E_j)$, $j=1,2,3$, be measurable spaces, each equipped with a distinguished point $o_j\in E_j$. 
For each $j$, define the punctured space $E_j^{o}:=E_j\setminus\{o_j\}$ and the associated $\sigma$-field $\mathcal E_j^{o}:=\{A\cap E_j^{o}: A\in\mathcal E_j\}$.  
As before, write the trivariate product space as $(E_{123},\mathcal E_{123})$, and denote by $(E_{123}^{o},\mathcal E_{123}^{o})$ the punctured product space, where
\[
E_{123}^{o}:=E_{123}\setminus\{(o_1,o_2,o_3)\},
\qquad 
\mathcal E_{123}^{o}:=\{A\cap E_{123}^{o}: A\in\mathcal E_{123}\}.
\]
Similarly, for $1\le j<k\le 3$, we denote the bivariate product spaces and their punctured versions by $(E_{jk},\mathcal E_{jk})$ and $(E_{jk}^{o},\mathcal E_{jk}^{o})$, respectively. Note that $E_{123}^o$ and $E_{jk}^o$ should be distinguished from $E_1^o\times E_2^o\times E_3^o$ and $E_{j}^o\times E_k^o$, respectively.

Next, suppose $\Lambda^o_{123}$ is a measure on the punctured space $(E_{123}^o,\cl{E}_{123}^o)$.   Denote the induced marginal measures on the non-punctured spaces $(E_j,\cl{E}_j)$ by
\[
\Lambda_j(A)
:=
\Lambda^o_{123}\big(\{(y_1,y_2,y_3)\in E_{123}^o:\; y_j\in A\}\big),
\quad A\in\mathcal E_j,
\quad j=1,2,3,
\]
and those  on the non-punctured spaces $(E_{jk},\cl{E}_{jk})$ by
 \[
\Lambda_{jk}(A)
:=
\Lambda^o_{123}\big(\{(y_1,y_2,y_3)\in E_{123}^o:\; (y_j,y_k)\in A\}\big),
\quad A\in\mathcal E_{jk},
\quad 1\le j<k\le 3.
\]
 We also denote the restrictions of these marginal measures to the corresponding punctured spaces by 
\[
\Lambda_j^o:=\Lambda_j\big|_{\mathcal E_j^o},
\qquad j=1,2,3,
\]
and
\[
\Lambda_{jk}^o:=\Lambda_{jk}\big|_{\mathcal E_{jk}^o},
\qquad 1\le j<k\le 3.
\]

We will be working with the following set of assumptions.
\begin{Ass}\label{Ass:puncture}~
\begin{enumerate}[label=(\roman*)]

\item  The measurable spaces $(E_j,\cl{E}_j)$, $j=1,2,3$, are  Borel. 

    \item Each marginal punctured space $E_j^o$ is localized by an increasing  sequence
$L_{h,j}\in\cl{E}_j^o$,  $h\in\bb{Z}_+$.  The  punctured product space $E_{123}^o$ is   localized by the sequence
\[
\Big(L_{h,1}\times E_2\times E_3\Big)
\;\cup\;
\Big(E_1\times L_{h,2}\times E_3\Big)
\;\cup\;
\Big(E_1\times E_2\times L_{h,3}\Big),
\qquad h\in\mathbb Z_+,
\]
and analogously for the punctured bivariate spaces $E_{jk}^o$, $1\le j<k\le3$;  
for example, $E_{12}^o$ is localized by
\[
\Big(L_{h,1}\times E_2\Big)\;\cup\;\Big(E_1\times L_{h,2}\Big),
\qquad h\in\mathbb Z_+.
\]

\item The measure  $\Lambda^o_{123}$ on $(E_{123}^o,\cl{E}_{123}^o)$ is locally finite.
\end{enumerate}
\end{Ass}

 \begin{Rem}\label{Rem:loc fin vs non}
  Under Assumption \ref{Ass:puncture},    since it is possible that
$
\Lambda_j(\{o_j\})=\infty$, $j=1,2,3$,
each of the marginal measures $\Lambda_j$   need not be  $\sigma$-finite.
Nevertheless, in view of Assumption \ref{Ass:puncture} (ii) and (iii), the restrictions
$
\Lambda_j^o 
$
are locally finite measures on the localized spaces $(E_j^o,\mathcal E_j^o)$, and thus also $\sigma$-finite. Similar comments apply to the bivariate marginal measures $\Lambda_{jk}$ and $\Lambda_{jk}^o$, $1\le j<k\le 3$.
 \end{Rem}

 We first prepare a key result before stating the main result of this section.

\begin{Lem}\label{Lem:uncond indep}
Suppose that Assumption~\ref{Ass:puncture}(i)--(ii) holds for $(E_1,\mathcal E_1)$, $(E_2,\mathcal E_2)$ with distinguished points $o_1\in E_1$, $o_2\in E_2$, and the punctured space $(E_{12}^o,\mathcal E_{12}^o)$.
Let $\xi_{12}^o$ be a Poisson point process on $E_{12}^o$ with locally finite mean measure $\Lambda_{12}^o$, and define the projections
\[
\xi_j(A):=\xi_{12}^o\big(\{(y_1,y_2)\in E_{12}^o:\ y_j\in A\}\big),\qquad A\in\mathcal E_j,\ j=1,2.
\]
Set $N:=\xi_{12}^o(E_{12}^o)$.  Then the conditional independence relation $\xi_1\perp\xi_2\mid N$ holds if and only if one of the following cases applies.
\begin{enumerate}[label=(\alph*)]
\item   $\Lambda_{12}^o(E_{12}^o)=0$.

\item   $\Lambda_{12}^o(y_1\neq o_1,\, y_2\neq o_2)=0$   and one of the following holds:
\begin{enumerate}[label=(b\arabic*)]
    \item  $\Lambda_{12}^o(y_1=o_1)=\Lambda_{12}^o(y_2=o_2)=\infty$;
    \item  $\Lambda_{12}^o(y_1=o_1)=\infty$, $\Lambda_{12}^o(y_2=o_2)=0$;
\item $\Lambda_{12}^o(y_1=o_1)=0$, $\Lambda_{12}^o(y_2=o_2)=\infty$.
\end{enumerate}

\item   $0<\Lambda_{12}^o(E_{12}^o)<\infty$,  $\Lambda_{12}^o(y_1=o_1)=0$ and
$\Lambda_{12}^o$ factorizes as $\Lambda_{12}^o(E_1^o\times E_2)^{-1} \Lambda_1^o\otimes\Lambda_2$ on $\mathcal E_1^o\otimes\mathcal E_2$.

\item $0<\Lambda_{12}^o(E_{12}^o)<\infty$, $\Lambda_{12}^o( y_2=o_2)=0$ and
$\Lambda_{12}^o$ factorizes as $\Lambda_{12}^o( E_1\times E_2^o)^{-1} \Lambda_1\otimes\Lambda_2^o$ on $\mathcal E_1\otimes\mathcal E_2^o$.
\end{enumerate}
 
\end{Lem}
\begin{Rem}
    Under the condition $\Lambda_{12}^o(y_1\neq o_1,y_2\neq o_2)=0$, note that 
    \[
    \Lambda_{12}^o(y_1=o_1)=    \Lambda_{12}^o(y_2\neq o_2),\qquad     \Lambda_{12}^o(y_2=o_2)=    \Lambda_{12}^o(y_1\neq o_1).
    \]
    So the subcases in (b) can be equivalently formulated using the measure values $ \Lambda_{12}^o(y_j\neq o_j)$, $j=1,2$.
\end{Rem}

\begin{proof}[Proof of Lemma \ref{Lem:uncond indep}]
~\\

 \noindent $\bullet$ The case  $\Lambda_{12}^o(E_{12}^o)=0$ is trivial since $\xi_1$, $\xi_2$ are both zero measures.

\noindent $\bullet$ Now consider the case  $\Lambda_{12}^o(E_{12}^o)=\infty$.

 In this case, we have $N=\infty$ a.s., and thus the conditional independence is equivalent to the unconditional independence $\xi_1\perp \xi_2$, which implies that
 \[
 \Cov\pp{\xi_1(A_1),\xi_2(A_2)}= \Lambda_{12}^o\pp{A_1\times A_2}=0 \text{ for any }A_j\in \cl{E}_j^o \text{ with }\Lambda_j^o(A_j)<\infty, \ j=1,2.
 \]
 Therefore, applying a $\sigma$-finite version of Dynkin's theorem (e.g., \cite[Theorem A.5]{LastPenrose2017}),   we have 
 \begin{equation}\label{eq:int zero}
 \Lambda_{12}^o(E_1^o\times E_2^o)=\Lambda_{12}^o(y_1\neq o_1,\, y_2\neq o_2)=0.
 \end{equation}

Under \eqref{eq:int zero}, the restrictions $\xi_1(E_1^o\cap\cdot)$ and
$\xi_2(E_2^o\cap\cdot)$ arise from independent restrictions of $\xi_{12}^o$
to the two axes. If $\Lambda_{12}^o(y_1=o_1)\in(0,\infty)$, then
$\xi_1(\{o_1\})=\xi_2(E_2^o)$ is a shared nondegenerate Poisson random
variable, so $\xi_1$ and $\xi_2$ cannot be independent. The symmetric
argument excludes $\Lambda_{12}^o(y_2=o_2)\in(0,\infty)$.
Since the total mass is infinite, only cases (b1)--(b3) remain.

Conversely, in case (b1), both origin masses
$\xi_1(\{o_1\})$ and $\xi_2(\{o_2\})$ equal $\infty$ a.s., so independence
follows from the independent restrictions away from the origins.
In cases (b2) and (b3), respectively, $\xi_1=\infty\delta_{o_1}$ and
$\xi_2=\infty\delta_{o_2}$ a.s., so independence also holds.
Since $N=\infty$ a.s., the required conditional independence follows.

\noindent $\bullet$ Suppose now $0<\Lambda_{12}^o(E_{12}^o)<\infty$.

Conditional on $N=n$, the point process $\xi_{12}^o$ has the same law as
$\sum_{i=1}^n\delta_{(Y_{i1},Y_{i2})}$, where the points are independent
with common law $\Lambda_{12}^o/\Lambda_{12}^o(E_{12}^o)$
\cite[Proposition 3.5]{LastPenrose2017}.
Since $\Prt{N=1}>0$, conditional independence of the projections given
$N$ implies independence of $Y_{11}$ and $Y_{12}$. The normalized
intensity assigns no mass to $(o_1,o_2)$, so
\[
0=\Prt{Y_{11}=o_1,\,Y_{12}=o_2}
=\Prt{Y_{11}=o_1}\Prt{Y_{12}=o_2}.
\]
Thus at least one marginal assigns no mass to its distinguished point,
yielding case (c) or (d).
Conversely, in either case the normalized intensity is a product measure.
Conditional on $N=n$, the coordinate arrays $(Y_{i1})_{i=1}^n$ and
$(Y_{i2})_{i=1}^n$ are therefore independent, and so are the projected
point measures. The case $n=0$ is trivial.

\end{proof}

Next, we turn to conditional independence. Denote by $\Lambda_{123^o}$   the restriction of $\Lambda_{123}^o$ to the product space $E_1\times E_2\times E_3^o$ equipped with  the product $\sigma$-field $\cl{E}_1\otimes \cl{E}_2\otimes \cl{E}_3^o$. 
Since $\Lambda_3^o$ is $\sigma$-finite, by disintegration (e.g., \cite[Theorem 3.4(iii)]{kallenberg:2021:foundations}), there exists a probability   kernel $\Lambda_{12\mid 3^o}(\cdot\mid \cdot):\cl{E}_{12} \times E_3^o\to [0,1] $, such that
\begin{equation}\label{eq:kernel 12 3o}
\Lambda_{123^o}(A) = \int \mbf{1}_A(y_1,y_2,y_3) \Lambda_3^o(dy_3)  \Lambda_{12\mid 3^o}(dy_1 dy_2|y_3), \quad A\in  \cl{E}_1\otimes \cl{E}_2\otimes \cl{E}_3^o.
\end{equation}
We also introduce the marginal probability kernels
\begin{equation}\label{eq:kernel 13o 23o}
\Lambda_{1\mid 3^o}(\cdot \mid \cdot )=\Lambda_{12\mid 3^o}(\cdot \times E_2 \mid \cdot ), \quad \Lambda_{2\mid 3^o}(\cdot \mid \cdot )=\Lambda_{12\mid 3^o}( E_1\times \cdot  \mid \cdot ).
\end{equation}

Let $\xi_{123}^o$ be a Poisson point process on $E_{123}^o$ with mean measure $\Lambda_{123}^o$.  
Define the projected Poisson point process on $E_j$ by
\[
\xi_j(A)
:=
\xi_{123}^o\big(\{(y_1,y_2,y_3)\in E_{123}^o:\; y_j\in A\}\big),
\quad A\in\mathcal E_j,
\quad j=1,2,3.
\]
Note that each $\xi_j$ is a Poisson point process with mean measure $\Lambda_j$ on the non-punctured marginal space $E_j$. 
We emphasize that, in view of Remark~\ref{Rem:loc fin vs non}, each restricted marginal point process 
$$
\xi_j^o:= \xi_j\big|_{\mathcal E_j^{o}}
$$ 
 has a locally finite mean measure $\Lambda_j^o$, and is thus locally finite a.s. 

Now we synthesize Proposition~\ref{Pro:cond indep prod} and 
Lemma~\ref{Lem:uncond indep} to characterize conditional independence 
$\xi_1 \perp \xi_2 \mid \xi_3$ on the punctured product space 
$\pp{E_{123}^o, \cl{E}_{123}^o}$. 
However, in  the result statement, we shall exclude the ``non-explosive'' cases~(c) and~(d) of Lemma~\ref{Lem:uncond indep}, a choice
consistent with the framework adopted in \cite{engelke2025graphical}, 
although these cases can be incorporated with suitable modifications. 
Moreover, under such a restriction, we can formulate a  characterization of conditional independence in terms of a simple functional representation at the level of 
the enumerated points of the Poisson point process. To describe such a representation, we introduce the following measure on $E_{123}^o$: 

\begin{equation}\label{eq:eta measure}
\begin{aligned}
\Lambda^{\perp}_{12;3}(A)
:={}&\Lambda^o_{123}\bigl(A\cap\{y_2=o_2,\,y_3=o_3\}\bigr)\\
&+\Lambda^o_{123}\bigl(A\cap\{y_1=o_1,\,y_3=o_3\}\bigr)\\
&+\Lambda_3^o\bigl(\{y_3\in E_3^o:(o_1,o_2,y_3)\in A\}\bigr),
\end{aligned}
\end{equation}

where $ A\in \cl{E}_{123}^o$.

\begin{Pro}\label{Pro:cond indep punc}
    Suppose Assumption \ref{Ass:puncture} holds, and $\xi_{123}^o$ is a Poisson point process with mean measure $\Lambda_{123}^o$ on the punctured product space $\pp{E_{123}^o,\cl{E}_{123}^o}$. Assume in addition that $\Lambda_{123}^o\pp{E_{123}^o}=\infty$, and that 
    \begin{equation}\label{eq:explosion cond}
    \begin{gathered}
        \Lambda^o_{123}\pp{y_3=o_3}\in\{0,\infty\},\\
        \Lambda^o_{123}\pp{y_j=o_j, y_3=o_3}\in \{0,\infty\},\quad j=1,2.
    \end{gathered}
    \end{equation}
  The following statements are equivalent.
   \begin{enumerate}[label=(\alph*)]
     \item  The conditional independence relation  $\xi_1\perp \xi_2 \mid \xi_3$ holds.
     \item The product measure factorization
     \begin{equation}\label{eq:punctured factorization}
     \Lambda_{12\mid 3^o}(\cdot \mid y_3)= \Lambda_{1\mid 3^o}(\cdot\mid y_3)\otimes \Lambda_{2\mid 3^o}(\cdot\mid y_3)
     \end{equation}
     holds for $\Lambda_3^o$-a.e.\ $y_3\in E_3^o$. Moreover,
     \begin{equation}\label{eq:punctured face zero}
          \Lambda_{123}^o(y_1\neq o_1,\, y_2\neq o_2,y_3=o_3)=0.
     \end{equation}
     \item Let $\sum_{i=1}^{\kappa} \delta_{ \pp{\eta_{i1},\eta_{i2},\eta_{i3}}}$ be an enumerated representation of  a Poisson point process on $E_{123}^o$ with mean measure $\Lambda_{12;3}^\perp$ as in \eqref{eq:eta measure}, where $\kappa$ is its total number of points. Then there exist   measurable functions $h_1:E_3\times [0,1]\to E_1$   and $h_2: E_3\times [0,1]\to E_2$ satisfying $h_j(o_3,\cdot)\equiv o_j$, $j=1,2$, so that 
     \begin{equation}\label{eq:func rep cond indep puncture}
     \xi_{123}^o \EqD 
    \sum_{i=1}^{\kappa} \delta_{\pp{h_1(\eta_{i3},\theta_{i1})+\eta_{i1},h_2(\eta_{i3},\theta_{i2})+\eta_{i2}, \eta_{i3}} },   
     \end{equation}
where $\pp{\theta_{i1}}_{i\in \bb{Z}_+}$ and $\pp{\theta_{i2}}_{i\in \bb{Z}_+}$  are  two i.i.d.\ $\mathrm{Uniform}(0,1)$ sequences which are mutually independent and independent of $\sum_{i=1}^{\kappa} \delta_{ \pp{\eta_{i1},\eta_{i2},\eta_{i3}}}$, and the plus sign $+$ is understood via $y_j+o_j=o_j+y_j=y_j$ for any $y_j\in E_j$, $j=1,2$.
 \end{enumerate}
\end{Pro}
 \begin{proof}
We decompose the Poisson point process $\xi_{123}^o$ into two independent components:
\[
\xi_{123}^o(\cdot)
=
\xi_{123}^o\bigl(\cdot \cap \{y_3 \neq o_3\}\bigr)
+
\xi_{123}^o\bigl(\cdot \cap \{y_3 = o_3\}\bigr).
\]
By \eqref{eq:explosion cond}, $\xi_3(\{o_3\})$ is deterministic,
equal to zero or infinity. Hence conditioning on $\xi_3$ is equivalent
to conditioning on $\xi_3^o$.

Suppose (a) holds. For $A_j\in\mathcal E_j^o$ with
$\Lambda_j^o(A_j)<\infty$, $j=1,2$, set
\[
f(y_3)
=
\Lambda_{12\mid3^o}(A_1\times A_2\mid y_3)
-\Lambda_{1\mid3^o}(A_1\mid y_3)
 \Lambda_{2\mid3^o}(A_2\mid y_3).
\]
Lemma~\ref{Lem:cond mean cov} and independence of
$\xi_{123}^o(\cdot\cap\{y_3\neq o_3\})$ and
$\xi_{123}^o(\cdot\cap\{y_3=o_3\})$ give
\[
0
=
\Cov\bigl(\xi_1(A_1),\xi_2(A_2)\mid\xi_3\bigr)
=
\int f\,d\xi_3^o
+\Lambda_{123}^o(A_1\times A_2\times\{o_3\}).
\]
The bound $|f|\le\Lambda_{1\mid3^o}(A_1\mid\cdot)\le1$ and
$\Lambda_1^o(A_1)<\infty$ give
$\int(|f|+f^2)\,d\Lambda_3^o<\infty$.
The Poisson integral above is deterministic, so its variance satisfies
\[
0=\Var\left(\int f\,d\xi_3^o\right)
=\int f^2\,d\Lambda_3^o.
\]
Thus $f=0$ $\Lambda_3^o$-a.e. and
\[
\Lambda_{123}^o(A_1\times A_2\times\{o_3\})=0.
\]
Applying Dynkin's theorem as in the proof of
Proposition~\ref{Pro:cond indep prod}, (a)$\Rightarrow$(b),
proves \eqref{eq:punctured factorization}.
Then
\[
\Lambda_{123}^o(y_1\neq o_1,\,y_2\neq o_2,\,y_3=o_3)
=\lim_{h\to\infty}
\Lambda_{123}^o(L_{h,1}\times L_{h,2}\times\{o_3\})=0,
\]
which proves \eqref{eq:punctured face zero}.

Conversely, \eqref{eq:punctured factorization} and
Proposition~\ref{Pro:cond indep prod}, (b)$\Rightarrow$(a),
give conditional independence between the projections of
$\xi_{123}^o(\cdot\cap\{y_3\neq o_3\})$ onto $E_1$ and $E_2$,
given $\xi_3^o$.
Equations \eqref{eq:punctured face zero} and \eqref{eq:explosion cond}
allow us to apply Lemma~\ref{Lem:uncond indep}(a)--(b) to
$\xi_{123}^o(\cdot\cap\{y_3=o_3\})$, viewed as a point process on
$E_{12}^o$. Its projections onto $E_1$ and $E_2$ are therefore independent.
Independence of
$\xi_{123}^o(\cdot\cap\{y_3\neq o_3\})$ and
$\xi_{123}^o(\cdot\cap\{y_3=o_3\})$ therefore yields
$\xi_1\perp\xi_2\mid\xi_3$, proving (a).

Under (b), Proposition~\ref{Pro:cond indep prod}(c), applied to
$\xi_{123}^o(\cdot\cap\{y_3\neq o_3\})$, together with the independent Poisson
point processes whose mean measures are the first two terms in
\eqref{eq:eta measure}, gives \eqref{eq:func rep cond indep puncture}.
Conversely, comparing mean measures in
\eqref{eq:func rep cond indep puncture} gives
\eqref{eq:punctured factorization} and
\eqref{eq:punctured face zero}.
\end{proof}

 \begin{Rem}
 The functional representation at the point level in 
\eqref{eq:func rep cond indep puncture} can be compared with a similar representation that 
arises in the context of extremal structural causal models 
\cite{fang2025structural}. In particular, when interpreting 
$y_1 \in E_1$, $y_2 \in E_2$, and $y_3 \in E_3$ as three variables, 
the representation \eqref{eq:func rep cond indep} corresponds to the 
structural equation associated with the directed graphical model 
$1 \leftarrow 3 \rightarrow 2$, which encodes the conditional 
independence relation $y_1 \perp y_2 \mid y_3$.
 \end{Rem} 

\section{General dimensions}\label{Sec:graph}

We begin by describing a setup that   extends Assumption~\ref{Ass:puncture} to higher-dimensional settings. Let $V=\{1,\ldots,d\}$, where $d\in\bb{Z}_+$, and let $(E_v,\cl{E}_v)$, $v\in V$, be measurable spaces, each equipped with a distinguished point $o_v\in E_v$.
For $A \subset V$, define the product space $E_A=\prod_{v \in A} E_v$ and the corresponding punctured space $E_A^o = E_A \setminus \{o_A\}$, where $o_A := (o_v)_{v \in A}$. We denote the induced product $\sigma$-field on $E_A$ by $\cl{E}_A$, and define
$
\cl{E}_A^o = \{G \cap E_A^o :\, G \in \cl{E}_A\}.
$
When $A=\{v\}$, we simply write $E_v$ and $E_v^o$. Suppose $\Lambda_V^o$ is a measure on $(E_V^o,\cl{E}_V^o)$.  For $A \subseteq V$ and $y \in E_V$, we write 
$y_A = (y_v)_{v \in A}$ for the corresponding subvector.
We shall work with the following set of assumptions.
\begin{Ass}\label{Ass:puncture graph}~
\begin{enumerate}[label=(\roman*)]

\item The measurable spaces $(E_v,\cl{E}_v)$, $v\in V$, are Borel.

\item Each marginal punctured space $E_v^o$ is localized by an increasing sequence $(L_{h,v})_{h\in\bb{Z}_+}$ with $L_{h,v}\in\cl{E}_v^o$. For every nonempty subset $A \subset V$, the punctured product space $E_A^o$ is localized by the sequence
\[
L_{h,A}
:= \bigcup_{a\in A} \left( L_{h,a} \times E_{A\setminus\{a\}} \right),
\qquad h\in\bb{Z}_+.
\]

\item The measure $\Lambda_V^o$ on $(E_V^o,\cl{E}_V^o)$ is locally finite.

  \item For all disjoint subsets $A,B\subset V$ with $A\neq\emptyset$,
\[
\Lambda_V^o \left(  y_A\neq o_A,\ y_B=o_B \right)\in\{0,\infty\}.
\]
\end{enumerate}
\end{Ass}

\begin{Rem}\label{Rem:localization consistent and E1}
The localization scheme satisfies the consistency property: for disjoint subsets
$A,B\subset V$ and all $h\in\bb{Z}_+$, we have
$$
L_{h,A\cup B}
=
\bigl(L_{h,A} \times E_B\bigr)
\cup
\bigl(E_A \times L_{h,B}\bigr).
$$
Note also that  condition (iv) implies $\Lambda_V^o\pp{E_{V}^o}\in\{0,\infty\}$. 

As pointed out by a referee,
Assumption~\ref{Ass:puncture graph}(ii) specifies a particular product
localization. The measurable sets contained in some $L_{h,V}$ form
an ideal, meaning a family closed under measurable subsets and finite
unions. This is the first product ideal generated by the coordinate
bounded sets; see
\cite[Section~2.3, Definition~2.3.17 and Example~2.3.19]{basrak2026foundations}.
\end{Rem}

The following lemma links  condition (iv) in Assumption \ref{Ass:puncture graph} to the explosiveness condition mentioned in \eqref{eq:E1}.
\begin{Lem}\label{Lem:E1 equiv}
Assumption \ref{Ass:puncture graph} (iv) holds if and only if
for every $D\subset V$ and every $d\in D$, we have
\begin{equation}\label{eq:E1 abstract}
\Lambda_V^o\!\left(  y_d\neq o_d,\ y_{V\setminus D}=o_{V\setminus D} \right)\in\{0,\infty\}.
\end{equation}
\end{Lem}

\begin{proof}
Suppose Assumption \ref{Ass:puncture graph} (iv) holds. Then for any $D\subset V$ and $d\in D$,
taking $A=\{d\}$ and $B=V\setminus D$ gives \eqref{eq:E1 abstract}.

Conversely, assume \eqref{eq:E1 abstract} holds.
Let $A,B\subset V$ be disjoint with $A\neq\emptyset$. Then
\[
\{y_A\neq o_A,\ y_B=o_B\}
=
\bigcup_{a\in A}\{y_a\neq o_a,\ y_B=o_B\}.
\]
Setting in \eqref{eq:E1 abstract}  $D=V\setminus B$, we infer that for each $a\in A\subset D$, we have
$$
\Lambda_V^o\!\left(y_a\neq o_a,\ y_B=o_B\right)\in\{0,\infty\}.
$$
Then the conclusion Assumption \ref{Ass:puncture graph} (iv) follows.
\end{proof}

\begin{Eg}
As a natural extension of the setup in \cite{engelke2025graphical}, 
where each variable is scalar-valued, one may consider a vector-valued 
framework in which $E_v = \mathbb{R}^{d_v}$ with $d_v \in \mathbb{Z}_+$ 
for each $v \in V$. We take $o_v$ to be the origin of 
$\mathbb{R}^{d_v}$ and define the punctured space 
$E_v^o := \mathbb{R}^{d_v} \setminus \{o_v\}$, for which we introduce the localization sequence 
\[
L_{h,v} = \{ y \in \mathbb{R}^{d_v} : \|y\| > 1/h \}, 
\quad h \in \mathbb{Z}_+,
\]
where $\|\cdot\|$ denotes an arbitrary norm on 
$\mathbb{R}^{d_v}$.   
The space $E_v=\bb{R}^{d_v}$ may also be replaced by the orthant  $[0,\infty)^{d_v}$, a natural choice in extreme value theory.

In these settings, suppose further that $\Lambda_V^o$ is an exponent
measure, so that for some $\alpha>0$,
\begin{equation}\label{eq:scaling}
\Lambda_V^o(tG)=t^{-\alpha}\Lambda_V^o(G),
\qquad t>0,\quad G\in\cl E_V^o.
\end{equation}
Assumption~\ref{Ass:puncture graph}(iv) is then automatic: its
coordinate-face sets $\{y_A\neq o_A,\,y_B=o_B\}$ are invariant under
positive scaling and therefore have mass either zero or infinity.

For nonempty disjoint $A,B,C$ forming a partition of $V$, suppose
$A\perp B\mid C$. Applying Proposition~\ref{Pro:cond indep punc} with
$(E_1,E_2,E_3)=(E_A,E_B,E_C)$ gives measurable representing maps
\[
h_A:E_C\times[0,1]\to E_A,\qquad
h_B:E_C\times[0,1]\to E_B.
\]
We further point out that due to the homogeneity of $\Lambda_V^o$  in \eqref{eq:scaling}, these maps can be chosen homogeneous
of degree one in their first argument:
\[
h_j(t y_C,u)=t h_j(y_C,u),
\qquad t>0,\quad y_C\in E_C,\quad u\in[0,1],\quad j\in\{A,B\}.
\]
For the details of the construction, see
\cite[proof of Theorem~2.10]{fang2025structural}.

\end{Eg}

We now extend Definition \ref{Def:Lambda cond indep} to the abstract setup of Assumption \ref{Ass:puncture graph}.
 In analogy with \eqref{eq:R test class}, define
\begin{equation}\label{eq:R test class abstract}
\mathcal{R}\left(\Lambda_V^o\right)
=\pc{R=\prod_{v\in V} R_v : 
R_v \in  \cl{E}_v ,\   \Lambda^o_V(R)>0,\ R \text{ is bounded}},
\end{equation}
where recall that $R$ being bounded means    $R\subset L_{h,V}$ for some $h$.

\begin{Def}\label{Def:cond ind abs}
Let $A,B,C\subseteq V$ be disjoint sets forming a partition of $V$.
We say that $\Lambda_V^o$ admits the conditional independence relation
\[
A \perp B \mid C
\]
if, for every $R\in\mathcal{R}(\Lambda_V^o)$ in  \eqref{eq:R test class abstract}, the  probabilistic
conditional independence
\[
Y_A \perp Y_B \mid Y_C
\]
holds, where $Y=(Y_v)_{v\in V}$ is an $R$-valued random vector with law
\begin{equation}\label{eq:P_R law}
P_R(G)
:=
\frac{\Lambda_V^o(G)}{\Lambda_V^o(R)},
\qquad
G\in\cl{E}_V,\; G\subseteq R.
\end{equation}

If $C=\emptyset$,  we understand the relation as unconditional independence $A\perp B$.
If $A=\emptyset$ or $B=\emptyset$, the conditional or unconditional independence relation is understood to hold 
trivially.

If $A,B,C$ do not form a partition of $V$, the definition is applied to
$\Lambda_D^o$, where $D=A\cup B\cup C$ and $\Lambda_D^o$ denotes the
marginal measure induced by $\Lambda_V^o$ on $E_D^o$.
\end{Def}

Next, we proceed to provide equivalent characterizations of Definition \ref{Def:cond ind abs}, for which we prepare some notation.
For $v\in V$ and $h\in\bb{Z}_+$, define the   sets
\[
R_{h,v}
:=
\{y\in E_V:\; y_v\in L_{h,v}\}
=
L_{h,v}\times E_{V\setminus\{v\}}.
\]
For $\Lambda_V^o(R_{h,v})>0$, let $P_{R_{h,v}}$ denote the normalized probability measure induced by
$\Lambda_V^o$ on $R_{h,v}$ as in \eqref{eq:P_R law}.

\begin{Thm}\label{Thm:characterization cond indep}
Suppose Assumption \ref{Ass:puncture graph} holds.
Let $A,B,C\subseteq V$ be disjoint sets forming a partition of $V$. Suppose $A\neq \emptyset$ and $B\neq \emptyset$ (otherwise, all of the statements below are understood to hold trivially).
The following statements are equivalent:

\begin{enumerate}[label=(\alph*)]

\item The conditional independence relation $A\perp B\mid C$ holds in the
sense of Definition~\ref{Def:cond ind abs}.

\item The probabilistic conditional independence $Y_A\perp Y_B\mid Y_C$
holds with $Y\sim P_{R_{h,v}}$  for all $v\in V$ and $h\in\bb{Z}_+$
such that $\Lambda_V^o(R_{h,v})>0$, where if $C=\emptyset$, the relation is understood as  unconditional independence.

\item If $C\neq\emptyset$, the probabilistic conditional independence $Y_A\perp Y_B\mid Y_C$
holds for all $h\in\bb{Z}_+$ and $c\in C$
such that $\Lambda_V^o(R_{h,c})>0$, with $Y\sim P_{R_{h,c}}$, and
additionally
\begin{equation}\label{eq:A perp B C=0}
\Lambda_V^o(y_A\neq o_A,\; y_B\neq o_B,\; y_C=o_C)=0.
\end{equation}
If $C=\emptyset$, we  have  
\begin{equation}\label{eq:A perp B}
\Lambda_V^o(y_A\neq o_A,\; y_B\neq o_B)=0.
\end{equation}
\item If $C\neq \emptyset$,  for $\Lambda_C^o$-a.e.\ $y_C\in E_C^o$, we have the following  product measure factorization:
\[
\Lambda_{AB\mid C^o}(\cdot\mid y_C)
=
\Lambda_{A\mid C^o}(\cdot\mid y_C)
\otimes
\Lambda_{B\mid C^o}(\cdot\mid y_C),
\]
where 
$
\Lambda_{AB\mid C^o}(\cdot \mid \cdot)
:\cl{E}_{A\cup B} \times E_C^o \to [0,1]
$,
$
\Lambda_{A\mid C^o}(\cdot \mid  \cdot )
:\cl{E}_A \times E_C^o \to [0,1]$ and $  \Lambda_{B\mid C^o}(\cdot \mid \cdot  )
:\cl{E}_B \times E_C^o \to [0,1]
$
are  probability kernels defined similarly as those in \eqref{eq:kernel 12 3o} and \eqref{eq:kernel 13o 23o},
 and additionally, the relation \eqref{eq:A perp B C=0} holds.  If $C=\emptyset$, we have \eqref{eq:A perp B}.

\item Let $\xi_V^o$ be a Poisson point process on $(E_V^o,\cl{E}_V^o)$ with
mean measure $\Lambda_V^o$.
  We have
the following probabilistic conditional independence:
\[
\xi_A \perp \xi_B \mid \xi_C,
\]
where for $I\subset V$,   
$
\xi_I(\cdot)
:=
\xi_V^o(y_I\in\cdot)
$,   and the relation above is understood as unconditional independence if $C=\emptyset$.
\end{enumerate}
\end{Thm}
\begin{proof}
We assume $\Lambda_V^o(E_V^o) = \infty$. The only remaining case 
(see Remark~\ref{Rem:localization consistent and E1}), namely 
$\Lambda_V^o(E_V^o) = 0$, is trivial.

The proofs of the equivalences among (a)--(d) follow 
arguments entirely analogous to those in \cite{engelke2025graphical}, 
with straightforward modifications arising from the abstract 
localization scheme. In particular, the equivalences (a)$\iff$(b) and (a)$\iff$(c) correspond 
to \cite[Theorem~4.1]{engelke2025graphical}, while the equivalence  (a)$\iff$(d) 
corresponds to \cite[Theorem~4.4]{engelke2025graphical}.

It remains to establish the equivalence (d) $\iff$ (e). 
If $C \neq \emptyset$, the result follows from 
Proposition~\ref{Pro:cond indep punc}. Indeed, condition (iv) in Assumption~\ref{Ass:puncture graph}
guarantees both conditions in \eqref{eq:explosion cond} for the grouped
coordinates $A,B,C$. 
If $C = \emptyset$, the claim follows from 
Lemma~\ref{Lem:uncond indep} applied to the case where the total 
measure (denoted there by $\Lambda_{12}^o(E_{12}^o)$) is infinite. 
\end{proof}
\begin{Rem}\label{Rem:localization invariance}
For fixed $\Lambda_V^o$ and distinguished points, the conditional
independence relation in Definition~\ref{Def:cond ind abs} is the same
for any two choices of localizing sequences satisfying
Assumption~\ref{Ass:puncture graph}. Indeed, by
Theorem~\ref{Thm:characterization cond indep}(e), both are equivalent to
the same conditional independence relation between the projected
Poisson point processes, although the individual laws $P_{R_{h,v}}$
depend on the chosen sequences.
\end{Rem}

\begin{Rem}\label{Rem:charac}
Theorem~\ref{Thm:characterization cond indep} is stated, for simplicity,
under the assumption that $A$, $B$, and $C$ form a partition of $V$.
The extension beyond the partition case is straightforward. In particular,
for statements (b)--(d), one replaces the role of $V$ by
$D:=A\cup B\cup C$ and the role of $\Lambda_V^o$ by the marginal measure
$\Lambda_D^o$.

For statement (e), no modification is required.
For each nonempty $I\in\{A,B,C\}$, the projections onto $E_I$ of
$\xi_V^o$ and $\xi_V^o(\cdot\cap\{y_D\neq o_D\})$ agree on $E_I^o$.
Their mean masses at $o_I$ are, respectively,
\[
\Lambda_V^o(y_I=o_I)
\quad\text{and}\quad
\Lambda_V^o(y_{D\setminus I}\neq o_{D\setminus I},\,y_I=o_I).
\]
Both belong to $\{0,\infty\}$ by
Assumption~\ref{Ass:puncture graph}(iv), so the origin masses are
deterministic. Thus the corresponding projections generate the same
$\sigma$-fields up to probability-null sets. The projection onto $E_D^o$ of
$\xi_V^o(\cdot\cap\{y_D\neq o_D\})$ is a Poisson point process with
mean measure $\Lambda_D^o$. Hence the equivalence with (e) follows
by applying the theorem with $D$ in place of $V$.
\end{Rem}

\begin{Cor}
Under Assumption \ref{Ass:puncture graph}, the conditional independence relation defined in Definition \ref{Def:cond ind abs} satisfies the semigraphoid axioms (L1)--(L4) in Section \ref{Sec:intro}.     
\end{Cor}
\begin{proof}
This follows from the same argument as \cite[Theorem 5.3]{engelke2025graphical}, noting   the equivalence pointed out by Lemma \ref{Lem:E1 equiv}.   
\end{proof}

\begin{Rem}\label{Rem:further extensions}
As pointed out by a referee, the product-ideal constructions in
\cite[Section~2.3]{basrak2026foundations} suggest possible extensions
beyond the localization in Assumption~\ref{Ass:puncture graph}(ii).
Establishing a corresponding Poisson point process characterization
would require suitable conditions on the coordinate projections and
a replacement for the decomposition according to whether $y_3=o_3$
used in the proof of Proposition~\ref{Pro:cond indep punc}.
\end{Rem}

\section*{AI use disclosure}
During the preparation of the manuscript, the authors used OpenAI Codex to assist with  checking
and revising proofs. The authors take full responsibility for the final
manuscript.

\printbibliography

\end{document}